\documentclass{article}

\usepackage{epsfig}
\usepackage{latexsym}

\usepackage{epsfig}
\usepackage{latexsym}
\usepackage{amsmath,amssymb}
\input{epsf}

\newtheorem{prelem}{{\bf Theorem}}

\newenvironment{theoremABC}{\begin{prelem}{\hspace{-0.5
               em}{\bf.}}}{\end{prelem}}

\newtheorem{theorem}{Theorem}[section]
\newtheorem{corollary}[theorem]{Corollary}
\newtheorem{definition}[theorem]{Definition}
\newtheorem{conjecture}[theorem]{Conjecture}
\newtheorem{claim}[theorem]{Claim}

\newtheorem{lemma}[theorem]{Lemma}
\newtheorem{observation}[theorem]{Observation}

\newenvironment{proof}{{\bf Proof.}}{\hfill\rule{2mm}{2mm}}
\newtheorem{remarka}[theorem]{Remark}
\newenvironment{remark}{\begin{remarka}\rm}{\hfill\rule{2mm}{2mm}\end{remarka}}

\def\deg {{\rm deg}}

\def\tr {{\rm tr}}

\title{Graph norms and Sidorenko's conjecture}
\author{
{\bf  Hamed Hatami} \\
{\small\it Department of Computer Science}\\
{\small University of Toronto} \\
{\small e-mail: hamed@cs.toronto.edu}}
\begin{document}
\maketitle

\tableofcontents

\begin{abstract}
Let $H$ and $G$ be two finite graphs. Define $h_H(G)$ to be the
number of homomorphisms from $H$ to $G$. The function $h_H(\cdot)$
extends in a natural way to a function from the set of symmetric
matrices to $\mathbb{R}$ such that for $A_G$, the adjacency matrix
of a graph $G$, we have $h_H(A_G)=h_H(G)$. Let $m$ be the number of
edges of $H$. It is easy to see that when $H$ is the cycle of length
$2n$, then $h_H(\cdot)^{1/m}$ is the $2n$-th Schatten-von Neumann
norm. We investigate a question of Lov\'{a}sz that asks for a
characterization of graphs $H$ for which the function
$h_H(\cdot)^{1/m}$ is a norm.

We prove that $h_H(\cdot)^{1/m}$ is a norm if and only if a
H\"{o}lder type inequality holds for $H$. We use this inequality to
prove both positive and negative results, showing that
$h_H(\cdot)^{1/m}$ is a norm for certain classes of graphs, and
giving some necessary conditions on the structure of $H$ when
$h_H(\cdot)^{1/m}$ is a norm. As an application we use the
inequality to verify a conjecture of Sidorenko for certain graphs
including hypercubes. In fact for such graphs we can prove
statements that are much stronger than the assertion of Sidorenko's
conjecture.

We also investigate the $h_H(\cdot)^{1/m}$ norms from a Banach space
theoretic point of view, determining their moduli of smoothness and
convexity. This generalizes the previously known result for the
$2n$-th Schatten-von Neumann norms.
\end{abstract}

%%%%%%%%%%%%%%%%%%%%%%%%%%%%%%%%%%%%%%%%%%%%%%%%%%%%%%%%%%%%%%%%%%%%%
\section{Introduction}

Let $H$ and $G$ be graphs. A {\sf homomorphism} from $H$ to $G$ is a
mapping $h:V(H) \rightarrow V(G)$ such that for each edge $\{u,v\}$
of $H$, $\{h(u),h(v)\}$ is an edge of $G$. Let $h_H(G)$ denote the
number of homomorphisms from $H$ to $G$. If $w$ is the adjacency
matrix of the graph $G$, then
\begin{equation}
\label{eq:tHgraph} h_H(G)=\sum_{x_a \in V(G)\ \forall a \in V(H)}
\left( \prod_{\{u,v\} \in E(H)} w(x_u,x_v) \right).
\end{equation}

We might also divide $h_H(G)$ by the total number of mappings from
$V(H)$ to $V(G)$ to obtain a normalized version:
\begin{equation}
\label{eq:tFunction} t_H(G) := \frac{h_H(G)}{|V(G)|^{|V(H)|}}.
\end{equation}
Thus $t_H(G)$ is the probability that a random mapping from $V(H)$
to $V(G)$ is a homomorphism.

The expression in the right hand side of (\ref{eq:tHgraph}) is quite
common. Such sums appear as {\sf Mayer} sums in classical
statistical mechanics, {\sf Feynman} sums in quantum field
theory~\cite{MR0239836} and {\sf multicenter} sums in quantum
chemistry~\cite{MR0128419}. In the present article when we study
$h_H(G)$, we usually think of $H$ as a fixed graph. In this case, as
it has been stated formally in Lemma 2.1 in~\cite{LovaszSzegedy06},
when $G$ is a sufficiently dense graph, $h_H(G)$ is a good
approximation for the number of copies of $H$ in $G$. This makes
understanding the behavior of $h_H$ one of the main objectives of
the extremal graph theory. Despite all the machinery that is
developed in recent years~\cite{MR2257396, MR2249277} and has been
applied successfully to some important questions \cite{Raz06}, still
there are many questions regarding these functions that are remained
unsolved. One of the important open questions in this area is  the
celebrated conjecture of Sidorenko~\cite{MR1138091}. The conjecture
says that for every graph $G$, and every \emph{bipartite} graph $H$
with  $m$ edges, we have
$$t_H(G) \ge t_{K_2}(G)^m,$$
where $K_2$ is the graph comprising two vertices and a single edge
between them. While the original motivation of this work was not to
study this conjecture, during the research we realized that our
results verify the conjecture for certain graphs including the
hypercubes. In fact for such graphs we can prove statements that are
surprisingly stronger than the assertion of Sidorenko's conjecture.
We discuss this more extensively in Section~\ref{sec:Sidorenko}.

Let us explain our main motivation. First we need to define
$h_H(\cdot)$ on a more general domain than graphs.

\begin{definition}
\label{def:WSymSpace} Consider an index set ${\cal I}$. Let ${\cal
WS}({\cal I})$ be the set of the symmetric real matrices indexed
over ${\cal I}$, i.e.
$${\cal WS}({\cal I})=\{w : {\cal I} \times {\cal I} \rightarrow \mathbb{R}:
\mbox{$w$ is symmetric} \}.$$
For a graph $H$ and $w \in {\cal WS}({\cal I})$, define
\begin{equation}
\label{eq:tHw} h_{H}(w):=\sum_{x_a \in {\cal I}\ \forall a \in V(H)}
\left(\prod_{\{u,v\} \in E(H)}w(x_u,x_v)\right).
\end{equation}
\end{definition}

 It turns out that for $C_4$, the cycle of size $4$, the function
$h_{C_4}(G)$ carries interesting information about $G$. For example
if $t_{C_4}(G)^{1/4}$ is close to $t_{K_2}(G)$, then $G$ ``looks
random'' in certain aspects~\cite{MR1054011}. Such graphs are
usually referred to as quasi-random graphs. Another interesting fact
about $h_{C_4}$ is that $h_{C_4}(\cdot)^{1/4}$ is a norm on ${\cal
WS}({\cal I})$. These observations belong to the same circle of
ideas  employed by Sz\'{e}meredi~\cite{MR0245555,MR0369312} to prove
his famous theorem on arithmetic progressions. In fact
Sz\'{e}meredi's regularity lemma, the main tool in the proof of his
theorem, roughly speaking says that every graph can be decomposed
into a few number of subgraphs such that most of them are
quasi-random (we refer the reader to Tao's survey~\cite{TaoMontreal}
for a precise formulation of the regularity lemma in terms of the
$h_{C_4}(\cdot)^{1/4}$ norm). Recently
Gowers~\cite{Gowers98,MR1844079} defined a hypergraph version of
this norm, and subsequently he~\cite{GowReg} and Nagle, R\"{o}dl,
Schacht, and Skokan~\cite{MR2198495, MR2198496, MR2069663}
independently established a hypergraph regularity lemma which easily
implies Szemer\'{e}di's theorem in its full generality, and even
stronger theorems such as Furstenberg-Katznelson's multi-dimensional
arithmetic progression theorem~\cite{MR1159329,MR1884430}, a result
that the only known proof for it at the time was through ergodic
theory~\cite{MR531279}. In fact arithmetic version of the Gowers
norm has interesting interpretations in ergodic theory, and has been
studied from that aspect~\cite{HostKra}. The discovery of this norm
led to a better understanding of the concept of quasi-randomness,
and provided strong tools. For example this norm plays an essential
role in the Green and Tao's proof~\cite{GreenTao} that the primes
contain arbitrarily long arithmetic progressions and the current
best bounds for the quantified version of the Szemer\'{e}di's
theorem is through the so called ``inverse theorems'' for these
norms~\cite{GreenTaoInverse,GreenTaoBoundI,GreenTaoBoundII,MR1844079}.

With all the known applications for the Gowers norms, it seems
natural to believe that studying $h_H(\cdot)^{1/|E(H)|}$ for graphs
other than $C_4$ might as well lead to some interesting
applications. In fact the main goal of this article is to pursue a
question of Lov\'{a}sz which asks for a characterization of graphs
$H$ for which the function $h_H(\cdot)^{1/|E(H)|}$ is a norm. We
prove both positive and negative results, showing that
$h_H(\cdot)^{1/|E(H)|}$ is a norm for certain classes of graphs, and
giving some necessary conditions on the structure of $H$ when
$h_H(\cdot)^{1/|E(H)|}$ is a norm. We hope that the application to
Sidorenko's conjecture promises discoveries of more applications in
the future.

We shall see in Section~\ref{sec:Schatten} that for $n>1$,
$h_{C_{2n}}(\cdot)^{1/2n}$ is the $2n$-th Schatten-von Neumann norm.
Probably after the $\ell_p$ spaces and the Banach lattices, the
Schatten-von Neumann spaces are the most well-studied normed spaces.
Therefore it seems natural to study the $h_H(\cdot)^{1/|E(H)|}$
norms from a Banach space theoretic aspect too. In this direction we
determine the moduli of convexity and smoothness of these spaces,
the two dual parameters that play a fundamental role in Banach space
theory. We discuss this further in Section~\ref{sec:Banach}.

\section{Definitions and main results}
\subsection{Notations and definitions \label{sec:notations}}
For $n \in \mathbb{N}$, let $[n]:=\{1,\ldots,n\}$. For two functions
$f,g: \mathbb{R} \rightarrow \mathbb{R}^+$, we write $f=o(g)$, if
and only if
$$\lim_{x \rightarrow \infty} f(x)/g(x)=0.$$
A {\sf graph} $G$ is a pair $(V,E)$ where $V$ is a finite set and
$E$ is a \emph{multi-set} (i.e. multiple copies of an element are
allowed) of the {\sf edges}, where every edge is an element of the
form $\{u,v\}$ with $u,v$ distinct elements in $E$. So we allow our
graphs to have multiple edges but no self-loops.

For a graph $G$, and an integer $k>0$, we denote by $G^{\between
k}$, the graph that is obtained by replacing every edge of $G$ by
$k$ multiple edges.

For a graph $G=(V,E)$, a set $S \subseteq V$ is called an {\sf
independent set} if there is no edge with both endpoints in $S$. For
the reasons that will be apparent soon we are mainly concerned about
the bipartite graphs. In graph theory, $G=(V,E)$ is called a
bipartite graph if $V$ can be partitioned into two disjoint
independent sets $V_1$ and $V_2$. We call the partition of $V$ into
$(V_1,V_2)$ a {\sf bipartization} of $G$. Note that disconnected
bipartite graphs have more than one bipartization. In this article
we use a different definition that fixes one specific bipartization
for $G$. So by a {\sf bipartite} graph we mean a triple $G=(X,Y;E)$,
where $X$ and $Y$ are two disjoint sets and $E$ is a multi-set of
the elements of $X \times Y$. Note that here we fix the
bipartization $(X,Y)$ as a part of the definition. Also note that
contrary to our definition of graphs, here every edge is an ordered
pair, and can be thought of as a directed edge from $X$ to $Y$.

Let $K_{m,n}$ be the complete bipartite graph, i.e. $K_{m,n}=(X,Y,X
\times Y)$ where $|X|=m$ and $|Y|=n$, and note that with our
definition unless $m=n$, $K_{m,n}$ is different from $K_{n,m}$. The
$n$-dimensional {\sf hypercube} $Q_n$ is the bipartite graph
$(X,Y;E)$ where $X$ is the set of elements of $\{0,1\}^n$ with an
even number of $1$'s in their coordinates, and $Y=\{0,1\}^n
\setminus X$. Moreover $(x,y) \in E$ if and only if $y$ differs only
in one coordinate from $x$.

For a bipartite graph $G$, the graph $G^{\between k}$ is defined in
a similar way to the general graphs. We also define a product for
bipartite graphs:

\begin{definition}
\label{def:biproduct} Let $G=(X,Y;E)$ and $H=(X',Y';E')$ be two
bipartite graphs. Then define $G \times^b H$, the bi-product of $G$
and $H$, as the bipartite graph with bipartization $\left(X \times
X', Y \times Y'\right)$ where the multiplicity of $((x,x'),(y,y'))$
in $G \times^b H$ is equal to the product of the multiplicities of
$(x,y) \in E$ and $(x',y') \in E'$.
\end{definition}

By a {\sf normed space} we mean a pair $(V,\|\cdot\|)$, where $V$ is
a vector space over $\mathbb{R}$ and $\|\cdot\|$ is a function from
$V$ to nonnegative reals satisfying

\begin{itemize}
\item[{\bf (i):}] $\|x\|=0$ if and only if $x=0$.

\item[{\bf (ii):}] $\|\lambda x\|=|\lambda| \|x\|$ for all $x \in V$ and $\lambda \in \mathbb{R}$.

\item[{\bf (iii):}] $\|x+y\|\le \|x\| + \|y\|$ for all $x,y \in V$.
\end{itemize}

We call $\|x\|$ the norm of $x$. A semi-norm is a function similar
to a norm except that it might not satisfy {\bf (i)}. A quasi-norm
is similar to a norm in that it satisfies the norm axioms, except
that {\bf (iii)} is replaced by $\|x+y\|\le K(\|x\|+\|y\|)$ for some
universal constant $K>0$.

\subsection{Graph norms}

As we discussed in the introduction $h_{C_4}(\cdot)^{1/4}$ is a
norm. Our main goal is to investigate a question of Lov\'{a}sz that
asks for a characterization of graphs $H$ for which the function
$h_H(\cdot)^{1/m}$ is a norm. Let $H$ be a nonbipartite graph and
$w_1=\left[\begin{array}{cc} 1 & 0\\ 0 & 1
\end{array} \right]$ and $w_2=\left[\begin{array}{cc} 0 & 1\\ 1 & 0
\end{array} \right]$. Then since $H$ is not bipartite we have
$h_H(w_2)=0$ and  we get $h_H(w_1)^{1/|E(H)|}+h_H(w_2)^{1/|E(H)|}
<h_H(w_1+w_2)^{1/|E(H)|}$. This shows that for our purposes it is
sufficient to restrict to the case where $H$ is bipartite. In this
case we can use a more general setting than ${\cal WS}$ and remove
the condition that $w$ is symmetric.

\begin{definition}
\label{def:WSpace} Consider two index sets ${\cal I}$ and ${\cal
J}$. Let
$${\cal W}({\cal I} \times {\cal J})=\{w : {\cal I} \times {\cal J} \rightarrow \mathbb{R}\},$$
and
$${\cal W}^+({\cal I} \times {\cal J})=\{w : {\cal I} \times {\cal J} \rightarrow \mathbb{R}^+\}.$$
For a bipartite graph $H=(X,Y;E)$ and $w \in {\cal W}({\cal I}
\times {\cal J})$, define
\begin{equation}\label{eq:asym} h_{H}(w):=\sum_{x_a \in {\cal I} \ \forall a \in X}
\ \sum_{y_b \in {\cal J} \ \forall b \in Y} \left(\prod_{(u,v) \in
E}w(x_u,y_v)\right);
\end{equation}
\begin{equation}
\|w\|_H:=|h_{H}(w)|^{1/|E|};
\end{equation}
\begin{equation}
\|w\|_{r(H)}:=h_{H}(|w|)^{1/|E|}.
\end{equation}
Furthermore let ${\cal W}_{H}({\cal I} \times {\cal J})$ and ${\cal
W}_{r(H)}({\cal I} \times {\cal J})$ respectively denote the set of
all $w \in {\cal W}$ with $\|w\|_H <\infty$ and
$\|w\|_{r(H)}<\infty$.
\end{definition}
\begin{remark}
Note that every bipartite graph $G=(X',Y';E')$ can be represented by
a zero-one matrix $w$ whose rows and columns are  indexed
respectively by elements of $X'$ and $Y'$, and every entry is equal
to one, if and only if its corresponding row and column are adjacent
in $G$. Then $h_H(w)$, as defined in (\ref{eq:asym}), is the number
of homomorphisms from $H$ to $G$ so that  $X$ is mapped into $X'$,
and $Y$ is mapped into  $Y'$.
\end{remark}

\begin{remark}
\label{rem:notation} From now on, when there is no ambiguity we drop
the variables from the subscript of sums. For example with this
notation, we allow $h_H(w)=\sum \prod_{(u,v) \in E(H)}w(x_u,y_v)$ or
even $h_H(w)=\sum \prod_{(u,v) \in E(H)}w$. We might also write
${\cal W}_H$ instead of ${\cal W}_H({\cal I} \times {\cal J})$.

Note that every $w \in {\cal W}_{H}({\cal I} \times {\cal J})$ can
be thought of as a matrix whose rows are indexed over ${\cal I}$ and
whose columns are indexed over ${\cal J}$. Let $w \in {\cal
W}_{H}({\cal I} \times {\cal J})$ and $w' \in {\cal W}_{H}({\cal J}
\times {\cal K})$. Then the matrix multiplication of $w$ to $w'$ is
defined. In order to distinguish between the matrix multiplication
and the pointwise multiplication we denote the former by $w \circ
w'$, and the latter by $ww'$. Moreover if $w \in {\cal W}({\cal I}
\times {\cal I})$, then $w^{\circ n}:=w \circ \ldots \circ w$, where
$w$ appears $n$ times in the right-hand side.
\end{remark}

We shall see below that neither $\|\cdot\|_H$ nor $\|\cdot\|_{r(H)}$
is always a norm. We have the following  observations:
\begin{observation}
\label{lem:observations} Let $H$ be a graph. Then
\begin{itemize}
\item[{\bf (i):}] If the function $\|\cdot\|_H$ is a semi-norm on ${\cal W}_{H}$, then
$\|\cdot\|_{r(H)}$ is a norm on ${\cal W}_{r(H)}$.

\item[{\bf (ii):}] If $H$ has a vertex of odd degree, then $\|\cdot\|_H$ is not always a
norm.
\end{itemize}
\end{observation}
\begin{proof}
Part {\bf (i)} is trivial. To prove {\bf (ii)}, note that
$\|w\|_H=0$ for
$$w=\left[\begin{array}{cc}1 & -1 \\
-1 & 1 \end{array}\right].$$
\end{proof}

Before continuing the discussion, let us first give a brief review
on Schatten-von Neumann norms, and see why $\|\cdot\|_{C_4}$ is a
special case of those norms.

\subsection{Schatten-von Neumann classes \label{sec:Schatten}}
Let $A$ be a real matrix. The $p$-th Schatten norm of $A$ is defined
as
$$\|A\|_{S_p} = (\tr(A^tA)^{p/2})^{1/p}.$$
Note that when $A$ is an $n \times n$ matrix, $\|A\|_{S_p}$ is just
the $\ell_p$ norm that is applied to the eigenvalues of $|A|=(A^t
A)^{1/2}$. This fact generalizes by the spectral theorem to the
infinite case. It is well-known (but not trivial) that
$\|\cdot\|_{S_p}$ is a norm. This can be deduced from
Theorem~\ref{thm:Trace} below due to Schatten and von
Neumann~\cite{MR0015674,MR0016533,MR0027127}. As it is mentioned in
Remark~\ref{rem:notation} we can consider the elements of ${\cal
W}({\cal I} \times {\cal J})$ as matrices. We state the theorem of
Schatten and von Neumann in this notation.
\begin{theoremABC}
\label{thm:Trace} Suppose that $1 \le p,q,r < \infty$ are such that
$\frac{1}{p}+\frac{1}{q}=\frac{1}{r}$. Let $v \in {\cal W}({\cal
I}_1, {\cal I}_2)$ and $w \in {\cal W}({\cal I}_2, {\cal I}_3)$.
Then
$$\|v \circ w\|_{S_r} \le \|v\|_{S_p} \|w\|_{S_q}.$$
\end{theoremABC}

Consider $C_{2n}$, a cycle of even length and a $w \in {\cal
W}({\cal I} \times {\cal J})$. Note that

\begin{eqnarray*}
\|w\|_{S_{2n}}^{2n}&=&\tr(w^t \circ w)^{\circ n}= \sum_{i \in {\cal
I}} \left((w^t \circ w)^{\circ n}\right)_{ii} \\ &=&
\sum_{i_0,i_2,\ldots, i_{2n-2} \in {\cal I}} \sum_{i_1,i_3,\ldots,
i_{2n-1} \in {\cal J}} w^t(i_0,i_1)
w(i_1,i_2) \ldots w^t(i_{2n-2},i_{2n-1}) w(i_{2n-1},i_0) \\
&=& \sum_{i_0,i_2,\ldots, i_{2n-2} \in {\cal I}}
\sum_{i_1,i_3,\ldots, i_{2n-1} \in {\cal J}} w(i_1,i_0) w(i_1,i_2)
\ldots w(i_{2n-1},i_{2n-2}) w(i_{2n-1},i_0) \\
&=& \|w\|_{C_{2n}}^{2n}.
\end{eqnarray*}
For further reading about the Schatten-von Neumann norms we refer
the reader to~\cite{MR1342297}.

\subsection{H\"{o}lder and weakly H\"{o}lder graphs \label{subsec:norm}}

The following is a corollary of Theorem~\ref{thm:Trace}.
\begin{corollary}
\label{cor:cycle} Let $k \ge 1$ be an integer, $V(C_{2k})=X \cup Y$
be the bipartization of $C_{2k}$, and $w_e \in {\cal W}({\cal I}
\times {\cal J})$ for every $e \in E(C_{2k})$. Then

$$\sum_{x_u \in {\cal I} \ \forall u \in X} \sum_{y_v \in {\cal J} \ \forall v \in Y}
\prod_{e=(u,v) \in E(C_{2k})} w_e(x_u,y_v) \le \prod_{e \in
E(C_{2k})} \|w_e\|_{C_{2k}}.$$
\end{corollary}
\begin{proof}
Let us identify $C_{2k}=(\{2i+1: 0 \le i \le k-1\},\{2i: 0 \le i \le
k-1\};E)$, where $$E=\{(1,0), (3,2),\ldots,(2n-1,2n-2)\} \cup
\{(1,2), (3,4),\ldots,(2n-1,0)\}.$$ We have
\begin{eqnarray*}
\sum \prod_{e=(u,v) \in E} w_e(x_u,y_v)  &=& \tr\left(w_{(1,0)}^t
\circ w_{(1,2)} \ldots \circ w_{(2k-1,2k-2)}^t \circ
w_{(2k-1,0)}\right) \\ &\le& \left\|w_{(1,0)}^t \circ w_{(1,2)}
\ldots \circ w_{(2k-1,2k-2)}^t \circ w_{(2k-1,0)}\right\|_{S_1}.
\end{eqnarray*}
Since $\frac{1}{2k}+\ldots+\frac{1}{2k}=\frac{1}{1}$, by repeatedly
applying Theorem~\ref{thm:Trace}, and noting that always
$\|w\|_{S_p}=\|w^t\|_{S_p}$, we get
$$\left\|w_{(1,0)}^t \circ w_{(1,2)} \ldots
\circ w_{(2k-1,2k-2)}^t \circ w_{(2k-1,0)}\right\|_{S_1} \le
\prod_{e \in E} \|w_e\|_{S_{2k}} = \prod_{e \in E}
\|w_e\|_{C_{2k}}.$$
\end{proof}

Corollary~\ref{cor:cycle} inspires us to have the following
definition.
\begin{definition}
\label{def:Holder} A bipartite graph $H$  is called
\begin{itemize}
\item {\bf H\"{o}lder:} If for every choice of
$\{w_e \in {\cal W}({\cal I} \times {\cal J}): e \in E(H)\}$ we have
\begin{equation}
\label{eq:Holder} \sum \prod_{e \in E(H)} w_e \le \prod_{e \in E(H)}
\|w_e\|_{H}.
\end{equation}

\item {\bf Weakly H\"{o}lder:} If for every choice of
$\{w_e \in {\cal W}({\cal I} \times {\cal J}): e \in E(H)\}$ we have
\begin{equation}
\label{eq:wHolder} \sum \prod_{e \in E(H)} w_e  \le \prod_{e \in
E(H)} \|w_e\|_{r(H)}.
\end{equation}
\end{itemize}
\end{definition}
Note that H\"{o}lder implies weakly H\"{o}lder, and by
Corollary~\ref{cor:cycle} cycles of even length are H\"{o}lder. We
prove the following theorem.

\begin{theorem}
\label{thm:Holder} A graph $H$ is H\"{o}lder if and only if
$\|\cdot\|_H$ is always a semi-norm. A graph $H$ is weakly
H\"{o}lder if and only if $\|\cdot\|_{r(H)}$ is always a norm.
\end{theorem}

Let us state our positive results.

\begin{theorem}
\label{thm:Hypercubes} We have the following:
\begin{itemize}
\item[{\bf (i):}] If $G$ and $H$ are both H\"{o}lder (weakly H\"{o}lder) then so
is $G \times^b H$.

\item[{\bf (ii):}] For every $m,n \ge 1$, the graph $K_{m,n}$ is weakly H\"{o}lder.
If both $m$ and $n$ are even then $K_{m,n}$ is H\"{o}lder.

\item[{\bf (iii):}] The hypercubes $Q_n$ are weakly H\"{o}lder.

\item[{\bf (iv):}] If $H$ is weakly H\"{o}lder, then
$H^{\between 2k}$ is H\"{o}lder for every integer $k>0$.
\end{itemize}
\end{theorem}

We also prove the following necessary conditions for a graph to be
weakly H\"{o}lder.

\begin{theorem}
\label{thm:criterion} If $G$ is weakly H\"{o}lder, then
\begin{itemize}
\item[{\bf (i):}] For every subgraph $H \subseteq G$, we have
$\frac{|E(H)|}{|V(H)|-1} \le \frac{|E(G)|}{|V(G)|-1}$.

\item[{\bf (ii):}] If $u$ and $v$ belong to the same part in the bipartization of $G$,
then $\deg(u)=\deg(v)$.
\end{itemize}
\end{theorem}
\begin{remark}
Theorem~\ref{thm:criterion} implies that among trees only $K_{1,n}$
are weakly H\"{o}lder. As we shall see in
Section~\ref{sec:criterion}, the proof of
Theorem~\ref{thm:criterion} shows that if a graph $G$ fails to
satisfy at least one of Theorem~\ref{thm:criterion} {\bf (i)} or
{\bf (ii)}, then the triangle inequality fails even if we restrict
ourselves to the symmetric matrices.
\end{remark}

\subsection{Sidorenko's conjecture \label{sec:Sidorenko}}

It is more natural to state Sidorenko's conjecture in a continuous
setting.

\begin{definition}
Consider two probability spaces $(M,{\cal F},\mu)$ and $(M',{\cal
F}',\nu)$. Let
$${\cal W}(\mu \times \nu)=\{w : \mu \times \nu \rightarrow \mathbb{R}:
\mbox{$w$ is measurable} \},$$
and for a bipartite graph $H=(X,Y;E)$ and $w \in {\cal W}(\mu \times
\nu)$, define
\begin{eqnarray}& h_{H}(w):=\int \prod_{(u,v) \in E}w(x_u,y_v) \prod_{t
\in X} d\mu(x_t)\prod_{r \in Y} d\nu(y_r); & \\
&\|w\|_{H}:=|h_{H}(w)|^{1/|E|}; &\\
&\|w\|_{r(H)}:=|h_{H}(|w|)|^{1/|E|}.&
\end{eqnarray}
Furthermore let ${\cal W}_{H}(\mu \times \nu)$ and ${\cal
W}_{r(H)}(\mu \times \nu)$ respectively denote the set of all $w \in
{\cal W}$ with $\|w\|_H <\infty$ and $\|w\|_{r(H)}<\infty$.
\end{definition}

The ${\cal W}_H(\mu \times \nu)$ spaces are related to ${\cal
W}_H({\cal I} \times {\cal J})$ spaces in the same way that $L_p$
spaces are related to $\ell_p$ spaces. So as one might expect it is
easy to see that all the results that are mentioned in the previous
sections hold for this setting as well. Now in this setting
Sidorenko's conjecture says that for every bipartite graph $H$ with
$m$ edges and every $w \in {\cal W}(\mu \times \nu)$, we have
$$h_{H}(|w|) \ge h_{K_2}(|w|)^m.$$
This simple-to-state conjecture is verified only for a handful of
graphs~\cite{MR1225933} including trees, even cycles, and complete
bipartite graphs. To see the importance of the conjecture note that
the case where $H$ is a path is equivalent to the Blakley-Roy
inequality~\cite{MR0184950} which has originally been proved by
Blakley and Roy using spectral techniques.

Sidorenko's conjecture has an interesting meaning: Fix a constant $0
\le p \le 1$. For an integer $n>0$, let $G(n,p)$ be a random graph
on $n$ vertices where each edge is present independently with
probability $p$. Let $\mu$ be the uniform measure on the vertices of
$G(n,p)$, and $w \in {\cal W}(\mu \times \mu)$ be its adjacency
matrix. Note that with high probability $h_{K_2}(w)=p \pm o(1)$, and
$h_{H}(w)=p^m \pm o(1)$. So roughly speaking Sidorenko's conjecture
says that for every bipartite graph $H$, among all graphs with fixed
number of vertices and edges, the random graphs asymptotically
contain the least number of copies of $H$.

Bal{\'a}zs Szegedy [private communication] pointed out to the author
that if $\|\cdot\|_{r(H)}$ is a norm, then Sidorenko's conjecture
holds for $H$. This can be easily seen from the convexity of norms.
But now that we have Theorem~\ref{thm:Holder}, in fact we can say
much more. Note that Sidorenko's conjecture can be reformulated as
the following.

\begin{conjecture}[Sidorenko's Conjecture]
Let $\mu$ be a probability measure. For every bipartite graph $H$
and every $w \in {\cal W}({\cal \mu} \times {\cal \mu})$, we have
$$\|w\|_{r(H)} \ge \|w\|_{r(K_2)}.$$
\end{conjecture}

We have the following result as a corollary to
Theorem~\ref{thm:Holder} which implies a stronger statement than of
Sidorenko's conjecture's when $\|\cdot\|_{r(H)}$ is a norm.

\begin{theorem}
\label{thm:strongSid} Let $\mu$ and $\nu$ be two probability
measures, and $H$ be a bipartite graph such that $\|\cdot\|_{r(H)}$
is a norm. Then for every subgraph $G \subseteq H$ and every $w \in
{\cal W}(\mu \times \nu)$ we have
$$\|w\|_{r(H)} \ge \|w\|_{r(G)}.$$
\end{theorem}
\begin{proof}
For $e \in E(G)$, define $w_e=w$, and for $e \in E(H) \setminus
E(G)$ define $w_e=1$. Since $\|\cdot\|_{r(H)}$ is a norm, by
Theorem~\ref{thm:Holder} we have
$$\int \left| \prod_{e \in E(H)} w_e \right| \le \prod_{e \in E(H)}
\|w_e\|_{r(H)}.$$
By our choice of $w_e$ we get $$h_G(|w|)=\int \left| \prod_{e \in
E(H)} w_e \right| \le \left(\prod_{e \in E(G)}
\|w\|_{r(H)}\right)\left( \prod_{e \in E(H)\setminus
E(G)}\|1\|_{r(H)}\right)=\|w\|_{r(H)}^{|E(G)|},$$ or equivalently
$\|w\|_{r(G)} \le \|w\|_{r(H)}$.
\end{proof}

\begin{remark}
Now by combining Theorem~\ref{thm:strongSid} and
Theorem~\ref{thm:Hypercubes} we see that $\|\cdot\|_{r(Q_n)}$ is an
increasing sequence of norms on ${\cal W}(\mu \times \nu)$. Note
that this is not true for ${\cal W}({\cal I} \times {\cal J})$.

Consider a probability measure $\mu$ and a symmetric function $w \in
{\cal W}(\mu \times \mu)$. In~\cite{Erdos}, Erd\"os and Simonovits
proved that for positive integers $n \le m$, we have
$\|w\|_{r(P_{2n})} \le \|w\|_{r(P_{2m})}$, where $P_k$ denotes the
path of length $k$. They furthermore conjectured
$\|w\|_{r(P_{2n-1})} \le \|w\|_{r(P_{2m-1})}$. This conjecture would
have been followed from Theorem~\ref{thm:strongSid}, if $P_{2m-1}$
was weakly H\"older, but Theorem~\ref{thm:criterion} shows that for
$m>2$, $P_{2m-1}$ is not weakly H\"older.
\end{remark}

\subsection{Banach Space properties \label{sec:Banach}} In this section we study the
Banach space properties of the graph norms.

The modulus of convexity of a Banach space $Y$ is a non-negative
function $\delta_Y$, defined for $\epsilon>0$ by
$$\delta_Y(\epsilon) = \inf \left\{ 1- \left\|\frac{x+y}{2}\right\|: x,y \in Y, \|x\|=\|y\|=1, \|x-y\| \ge \epsilon\right\}. $$
The modulus of smoothness of $Y$ is a function $\rho_Y$ defined for
$\epsilon>0$ by
$$\rho_Y(\epsilon) = \frac{1}{2}\sup \left\{ \left\|x+y\right\|+\left\|x-y\right\|-2: x,y \in Y, \|x\|=1, \|y\|=\epsilon\right\}. $$
Next theorem shows that if $H$ is a H\"{o}lder graph with $m$ edges,
then $\|\cdot\|_H$ has the same moduli of smoothness and convexity
as $\ell_m$. This was known~\cite{Tomczak} for $H=C_{2n}$ due to the
relation to the Schatten-von Neumann norms.
\begin{theorem}
\label{thm:moduli} There exist constants $C_m>0$ and $C'_m>0$ such
that the following holds. Let $H$ be a H\"{o}lder graph with $m$
edges, and ${\cal I}$ and ${\cal J}$ be two infinite sets. Then
\begin{equation}
\label{eq:thm:convexity} C_m \delta_{\ell_m} \le \delta_{{\cal W}_H}
\le \delta_{\ell_m},
\end{equation}
and
\begin{equation}
\label{eq:thm:smoothness} C'_m \rho_{\ell_m} \le \rho_{{\cal W}_H}
\le \rho_{\ell_m}.
\end{equation}
\end{theorem}
%
%Let $D=\{-1,1\}^n$ and let $\nu$ be the uniform probability measure
%on $D$. We denote by $\epsilon_n : D \rightarrow \{-1,1\}$, the
%$n$-th coordinate on $D$. Thus, the sequence $(\epsilon_n)$ is an
%i.i.d. sequence of symmetric Bernoulli $\{-1,1\}$-valued random
%variables. These random variables are usually called the Rademacher
%random variables.
%
%\begin{definition}
%\begin{enumerate}
%\item Let $1 \le p \le 2$. A normed space $B$ is called of type $p$
%if there exists a constant $T_p$, such that for all finite sequence
%$(x_i)$ in $B$,
%%
%\begin{equation}
%\label{eq:type} \left(\int \|\sum \epsilon_i x_i \|^2
%d\nu\right)^{1/2} \le T_p \left( \sum \|x_i\|^p \right)^{1/p}.
%\end{equation}
%
%\item  Let $2 \le q \le \infty$. A normed space $B$ is called of
%cotype $q$, if there exists a constant $C_p$, such that for all
%finite sequence $(x_i)$ in $B$,
%%
%\begin{equation}
%\label{eq:cotype} \left( \sum \|x_i\|^q \right)^{1/q} \le C_p
%\left(\int \|\sum \epsilon_i x_i \|^2 d\nu\right)^{1/2}.
%\end{equation}
%\end{enumerate}
%\end{definition}
%
%\begin{remark}
%Clearly if $p_1 \le p_2$, type $p_2$ implies type $p_1$ while cotype
%$p_1$ implies cotype $p_2$. Observe that every Banach space is of
%type $1$ and cotype $\infty$. For $1 \le p \le 2$, the $\ell_p$
%space is of type $p$ and nothing more, and cotype $2$. For $2 \le q
%<\infty$, the $\ell_q$ space is of cotype $q$ and nothing less, and
%of type $2$.
%\end{remark}

For the definitions of type and cotype of a Banach space we refer
the reader to~\cite{MR540367}. It is known
\cite{Figiel,FigielPisier} that the modulus of convexity of power
type $q \ge2$ implies cotype $q$, and the modulus of smoothness of
power type $1 < p \le 2$ implies type $p$. Thus
Theorem~\ref{thm:moduli} determines the type and cotype of ${\cal
W}_H$.

\begin{theorem}
Let $H$ be a graph with $m$ edges such that $\|\cdot\|_H$ is a norm
on ${\cal W}_H({\cal I} \times {\cal J})$. Then ${\cal W}_H$ is of
type $2$ and cotype $m$, and it is not of any cotype $q<m$ if ${\cal
I}$ and ${\cal J}$ are both infinite.
\end{theorem}
\begin{proof}
The type and cotype follows from the results of Figiel and Pisier
\cite{Figiel,FigielPisier}. It remains to show that if ${\cal I}$
and ${\cal J}$ are infinite, then ${\cal W}_H$ is not of cotype
$q<m$. But if ${\cal I}$ and ${\cal J}$ are infinite, then ${\cal
W}_H$ contains all finite dimensional subspaces of $\ell_m$ as
subspace, and thus it cannot be of cotype $q<m$.
\end{proof}

\section{Proofs}
\subsection{Proof of Theorem~\ref{thm:Holder}} Let us first
develop some tools. Let $w_1 \in {\cal W}({\cal I} \times {\cal J})$
and $w_2 \in {\cal W}({\cal I}' \times {\cal J}')$. Define $w_1
\otimes w_2 \in {\cal W}\left(({\cal I} \times {\cal I}') \times
({\cal J} \times {\cal J}')\right)$, as $[(x,x'),(y,y')] \mapsto
w_1(x,y)w_2(x',y')$. We also define $w^{\otimes k}=w \otimes \ldots
\otimes w$, where $w$ appears $k$ times in the right-hand side. We
have the following trivial observation.
\begin{lemma}
\label{lem:tensor} Let $H=(X,Y;E)$ be a bipartite graph, and $w_e
\in {\cal W}({\cal I} \times {\cal J})$ and $w'_e \in {\cal W}({\cal
I}' \times {\cal J}')$ for $e \in E$. Then
$$\sum \prod_{e \in H}w_e \otimes w'_e= \left(\sum \prod_{e \in H}w_e \right) \left(\sum \prod_{e \in H}w'_e \right).$$
\end{lemma}
Now with this lemma in hand we can prove Theorem~\ref{thm:Holder}
with the standard tensor power trick.

\noindent
\begin{proof}[Theorem~\ref{thm:Holder}]
Let $H$ be a H\"{o}lder graph with $m$ edges and $w_1, w_2 \in {\cal
W}_H({\cal I} \times {\cal J})$. Then by expanding $h_H(w_1+w_2)$
and applying (\ref{eq:Holder}) to each term, it is clear that
$$h_H(w_1+w_2) \le  (\|w_1\|_H+\|w_2\|_H)^{m}.$$
This proves that $\|\cdot\|_H$ is a semi-norm. Now suppose that $H$
is not H\"{o}lder. Then there exists $\{w_e \in {\cal W}({\cal I}
\times {\cal J}): e \in E(H)\}$, such that
$$\sum \prod_{e \in E(H)} w_e > \prod_{e \in E(H)} \|w_e\|_{H}.$$
After proper normalization we may assume that $\|w_e\|_H \le 1$, for
every $e \in E(H)$, and $\sum \prod_{e \in E(H)} w_e=c$, for some
$c>1$. Now by Lemma~\ref{lem:tensor}
\begin{eqnarray*}
\left\|\sum_{e \in E(H)} w_e^{\otimes 2n}\right\|_H^m &=& \sum_{x_u
\in {\cal I}, y_v \in {\cal J}} \prod_{e' \in E(H)} \left(\sum_{e
\in E(H)} w_e^{\otimes 2n}\right) \\ &=& \sum_{f:E(H) \rightarrow
E(H)} \left( \sum_{x_u \in {\cal
I}, y_v \in {\cal J}} \prod_{e \in E(H)} w_{f(e)}^{\otimes 2n}\right) \\
&=& \sum_{f:E(H) \rightarrow E(H)} \left( \sum_{x_u \in {\cal I},
y_v \in {\cal J}} \prod_{e \in E(H)} w_{f(e)} \right)^{\otimes 2n}\\
&\ge& \left(\sum_{x_u \in {\cal I}, y_v \in {\cal J}}\prod_{e \in E(H)} w_e\right)^{2n} =c^{2n},\\
\end{eqnarray*}
while for every $e \in E(H)$, by Lemma~\ref{lem:tensor} we have that
$$\left\|w_e^{\otimes 2n}\right\|_H=\left\|w_e\right\|_H^{2n} \le 1.$$
Thus for large enough $n$, the triangle inequality fails:
\begin{equation}
\label{eq:thm:Holder} \sum_{e \in E(H)} \left\|w_e^{\otimes
2n}\right\|_H \le m < c^{2n/m} \le \left\|\sum_{e \in E(H)}
w_e^{\otimes 2n}\right\|_H.
\end{equation}
The proof of the weakly H\"{o}lder case is similar.
\end{proof}

It is easy to see that (\ref{eq:thm:Holder}) implies the following
corollary.

\begin{corollary}
\label{cor:quasi} Let $H$ be a graph such that $\|\cdot\|_H$ is a
quasi-norm (semi-quasi-norm, respectively) on ${\cal W}_H$. Then
$\|\cdot\|_H$ is a norm (semi-norm respectively) on ${\cal W}_H$.
The same statement holds for $\|\cdot\|_{r(H)}$.
\end{corollary}

\subsection{Proof of Theorem~\ref{thm:Hypercubes}
\label{sec:hypercubes}}

{\bf (i):} Suppose that $G$ and $H$ are both H\"{o}lder with $m$ and
$m'$ edges, respectively. Consider $\{w_e \in {\cal W}({\cal I}
\times {\cal J}): e \in E(G \times^b H)\}$. We have
\begin{eqnarray}
\nonumber \sum \prod_{e \in E(G \times^b H)} w_e &=& \sum
\prod_{(a,b) \in E(G)} \prod_{(u,v) \in E(H)}
w_{[(a,u),(b,v)]}(x_{[a,u]},y_{[b,v]})
\\ \label{eq:biprod1} &\le& \prod_{(a,b) \in E(G)} \left|\sum \prod_{(a',b') \in E(G)} \prod_{(u,v) \in E(H)}
w_{[(a,u),(b,v)]}(x_{[a',u]},y_{[b',v]})\right|^{1/m} \\ \nonumber
&=& \prod_{(a,b) \in E(G)} \left|\sum \prod_{(u,v) \in E(H)}
\prod_{(a',b') \in E(G)} w_{[(a,u),(b,v)]}(x_{[a',u]},y_{[b',v]})\right|^{1/m} \\
\label{eq:biprod2} &\le& \prod_{(a,b) \in E(G)} \left|\prod_{(u,v)
\in E(H)} \left| \sum \prod_{(u',v') \in E(H)} \prod_{(a',b') \in
E(H)}
w_{[(a,u),(b,v)]}(x_{[a',u']},y_{[b',v']})\right|^{1/m'}\right|^{1/m} \\
\nonumber&=& \prod_{(a,b) \in E(G)} \prod_{(u,v) \in E(H)} \left|
\sum \prod_{(u',v') \in E(H)} \prod_{(a',b') \in E(H)}
w_{[(a,u),(b,v)]}(x_{[a',u']},y_{[b',v']})\right|^{1/mm'} \\
\nonumber &=& \prod_{[(a,u),(b,v)] \in E(G)}
\left\|w_{[(a,u),(b,v)]}\right\|_{G \times^b H},
\end{eqnarray}
where in (\ref{eq:biprod1}) and (\ref{eq:biprod2}) we applied
(\ref{eq:Holder}). The case of weakly H\"{o}lder is similar.

\noindent {\bf (ii):} The fact that $K_{1,2n}$ is H\"{o}lder follows
from the classical H\"{o}lder inequality. Indeed let
$X(K_{1,2n})=\{u\}$ and $Y(K_{1,2n})=\{v_1,\ldots,v_{2n}\}$. Then
\begin{eqnarray*}
\sum  \prod_{i=1}^{2n} w_{(u,v_i)} &=&\sum_{x_u \in {\cal I}}
\prod_{i=1}^{2n} \left(\sum_{y_{v_i} \in {\cal J}} w_{(u,v_i)}
\right) \\&\le& \prod_{i=1}^{2n} \left(\sum_{x_u \in {\cal I}}
\left( \sum_{y_{v_i} \in {\cal J}} w_{(u,v_i)}
\right)^{2n}\right)^{1/2n}=\prod_{i=1}^{2n}
\|w_{(u,v_i)}\|_{K_{1,2n}}.
\end{eqnarray*}
Similarly one can show that $K_{1,n}$ is weakly H\"{o}lder. Now the
assertion follows from {\bf (i)} and the fact that $K_{m,n}=K_{1,n}
\times^b K_{m,1}$.

\noindent {\bf (iii):} Note that $\|\cdot\|_{K_2^{\between 2k}}$ is
just the $\ell_{2k}({\cal I} \times {\cal J})$ norm. Now notice that
$G^{\between 2k}=G \times^{b} K_2^{\between 2k}$, and hence is
weakly H\"{o}lder. But it is easy to see that $\|.\|_{G^{\between
2k}}=\|\cdot\|_{r(G^{\between 2k})}$.

\noindent {\bf (iv):} Let $Q_n$ denote the $n$-dimensional
hypercube. We identify the vertices of a hypercube $Q_n$ with the
$0$-$1$ strings of length $n$, where two vertices are adjacent if
their strings differ in one bit. With this notation we can
concatenate two nodes $s \in V(Q_n)$ and $v \in V(Q_m)$ to obtain
the node $sv \in V(Q_{n+m})$. Note that $Q_n$ is bipartite. We use
the convention that $X(Q_n)$ is the set of vertices with an even
number of $1$'s in their strings, and $Y(Q_n)$ is the rest of the
vertices.

%For every vertex $v \in V(Q_n)$, denote by $V_v(Q_n) \in \{X,Y\}$,
%the set of all vertices that are in the same partition as $v$ in the
%bipartization of $Q_n$. Suppose that $w_{(u,v)} \in {\cal W}({\cal
%I} \times {\cal J})$. In the following we shall use the notation
%$w_{(v,u)}:=w_{(u,v)}^t \in {\cal W}({\cal J} \times {\cal I})$,
%i.e. $w_{(v,u)}(x,y):=w_{(u,v)}(y,x)$.

For all $u \in X$ and $v \in Y$, let $f_u :{\cal I} \rightarrow
\mathbb{R}$ and $g_v : {\cal J} \rightarrow \mathbb{R}$ be
functions, and for every edge $e \in Q_n$ let $w_e \in {\cal
W}({\cal I} \times {\cal J})$. We claim the following strengthening
of Theorem~\ref{thm:Hypercubes}:

\begin{claim}
\label{claim:hypercubes}
\begin{equation}
\label{eq:cube} \sum \left| \prod_{u \in X(Q_n), v \in Y(Q_n)}
f_u(x_u) g_v(y_v) \prod_{e=(a,b) \in E} w_e(x_a,y_b)\right|  \le
\prod_{e=(a,b) \in E} \left(\sum |R_e| \ \right)^{1/|E(Q_n)|},
\end{equation}
where for $e=(a,b)$,
$$R_e:=
\prod_{u \in X(Q_n), v \in Y(Q_n)} f_a(x_u)g_b(y_v) \prod_{(s,t) \in
E} w_e(x_s,y_t).$$
\end{claim}

If one substitutes $f_u=1$, $g_v=1$ for every $u \in X(Q_n)$, and $v
\in Y(Q_n)$, then Claim~\ref{claim:hypercubes} reduces to
Theorem~\ref{thm:Hypercubes} {\bf (iv)}. So it is sufficient to
prove the claim. Before proving Claim~\ref{claim:hypercubes} in its
general form we prove it for $n=2$ as a separate lemma. First notice
that without loss of generality we can assume that $f_u,g_v \ge0$
and $w_e \ge 0$ for every $u \in X(Q_n)$, $v \in Y(Q_n)$, and $e \in
E(Q_n)$, and drop the absolute value signs from the proof.

\begin{lemma}
\label{lem:base} Claim~\ref{claim:hypercubes} holds for $n=2$.
\end{lemma}
\begin{proof}
For an edge $e=(u,v) \in Q_2$, define
$$w'_e : {\cal I} \times {\cal J} \rightarrow \mathbb{R},$$
$$w'_e: (x,y)\mapsto \sqrt{f_u(x)}w_e(x,y)\sqrt{g_v(y)}.$$
Now since $Q_2$ is isomorphic to $C_4$, for $n=2$, we have
$$\mbox{L.H.S. of (\ref{eq:cube})}=\sum \prod_{e=(u,v) \in E(Q_2)} w'_e(x_u,y_v) \le
\prod_{e \in E(C_4)} \|w'_{e}\|_{C_4} =\prod_{e \in E(Q_2)}
\left(\sum R_e \right)^{1/4}.$$
\end{proof}

We now turn to the proof of Claim~\ref{claim:hypercubes} in its
general form. The proof is divided into several steps, so that
hopefully the main ideas can be distinguished from technicalities.
Let us first introduce some notation that helps us to keep the proof
short.

\begin{remark}
\label{rem:homs} Let $\phi:V(Q_n) \rightarrow V(Q_n)$ be such that
$\phi(u) \in X$ and $\phi(v) \in Y$ for every $u \in X$ and $v \in
Y$, and furthermore $(\phi(u),\phi(v)) \in E$ if $(u,v) \in E$.
Define
$$R_\phi=\prod_{u \in X(Q_n), y \in Y(Q_n)} f_{\phi(u)}(x_u)
g_{\phi(v)}(y_v) \prod_{(s,t) \in E(Q_n)}
w_{(\phi(s),\phi(t))}(x_s,y_t).$$
For example, for $e=(a,b)$, let $\phi_e$ be defined as
$\phi_e(u)=a$, if  $u \in X(Q_n)$ and $\phi_e(u)=b$, otherwise. Then
$R_e$ as it is defined in Claim~\ref{claim:hypercubes} is in fact
the same as $R_{\phi_e}$, and if we denote by $id.$ the identity map
from $V(Q_n)$ to itself, then Claim~\ref{claim:hypercubes} says that
\begin{equation}
\label{eq:cubeReform}  \sum |R_{id.}| \le \prod_{e \in
E(Q_n)}\left(\sum |R_{\phi_e}|\right)^ {1/|E(Q_n)|}.
\end{equation}
\end{remark}

 \noindent
\begin{proof}[Claim~\ref{claim:hypercubes}]
 We prove the claim by
induction. Before engaging in the calculations, let us explain the
intuition behind the proof. The variables $x_u, y_v$ assign some
values to the vertices. The product in the left-hand side of
(\ref{eq:cube}) is the product of the functions $f_u, g_v$ and $w_e$
where $f_u$ and $g_v$ depend only on the values that are assigned to
the vertices $u$ and $v$ respectively, and $w_e$ depends only on the
values that are assigned to the endpoints of $e$. The first step in
the proof is to group these functions together so that they can be
interpreted as the same product but for $Q_{n-1}$ instead of $Q_n$.
Then we can apply the induction hypothesis.

\noindent {\bf Step 1:} We regroup the product in the left-hand side
of~(\ref{eq:cube}) in the following way.
\begin{eqnarray}
\nonumber \lefteqn{\prod_{u \in X(Q_n), v \in Y(Q_n)} f_u g_v
\prod_{e \in E(Q_n)} w_e =} \\\label{eq:Ind1} && \prod_{u \in
X(Q_{n-1}), v \in Y(Q_{n-1})} (f_{0u} w_{(0u,1u)} g_{1u})(f_{1v}
w_{(1v,0v)} g_{0v}) \prod_{(s,t) \in E(Q_{n-1})}\left(
w_{(0s,0t)}w_{(1t,1s)}\right).
\end{eqnarray}
The left-hand side of (\ref{eq:Ind1}) is the product in the
left-hand side of~(\ref{eq:cube}), and the right-hand side
of~(\ref{eq:Ind1}) can be interpreted as the same product for
$Q_{n-1}$ but on different index sets in the following way: Let the
value assigned to the vertices $u \in X(Q_{n-1})$ and $v \in
Y(Q_{n-1})$ be the pair $[x_{0u},y_{1u}]$ and $[x_{1v},y_{0v}]$
respectively. Note that in the right-hand side of~(\ref{eq:Ind1}),
$f_{0u} w_{(0u,1u)} g_{1u}$ depends only on $[x_{0u},y_{1u}]$, and
$f_{1v} w_{(1v,0v)} g_{0v}$ depends only on $[x_{1v},y_{0v}]$, and
finally $w_{(0s,0t)}w_{(1t,1s)}$ depends only on the pair
$([x_{0s},y_{1s}], [x_{1t},y_{0t}])$.

More formally, to prove the claim for $n$ and ${\cal I} \times {\cal
J}$, we use the induction hypothesis for $n-1$ with the index set
$({\cal I} \times {\cal J}) \times ({\cal I} \times {\cal J})$.
Every vertex $v \in Q_{n-1}$ corresponds to two adjacent vertices
$0v$ and $1v$.  To use the induction hypothesis, for $u \in
X(Q_{n-1})$, define
$$f'_u: {\cal I} \times {\cal J} \rightarrow \mathbb{R}$$
$$f'_u: [x,y] \mapsto f_{0u}(x) w_{(0u,1u)}(x,y) g_{1u}(y).$$
For $v \in Y(Q_{n-1})$, define
$$g'_v: {\cal I} \times {\cal J} \rightarrow \mathbb{R}$$
$$g'_v: [x,y] \mapsto f_{1v}(x) w_{(1v,0v)}(x,y) g_{0v}(y).$$
and for $e=(u,v) \in E(Q_{n-1})$,
$$w'_e: ({\cal I} \times {\cal J}) \times ({\cal I} \times {\cal J}) \rightarrow \mathbb{R}$$
$$w'_e: ([x,y], [x',y']) \mapsto w_{(0u,0v)}(x,y')w_{(1v,1u)}(x',y).$$
Then
\begin{eqnarray}
\nonumber \lefteqn{\mbox{L.H.S. of (\ref{eq:cube})} = \mbox{R.H.S. of (\ref{eq:Ind1})}=} \\ &&\label{eq:Ind2} \sum \left|
\prod_{u \in X(Q_{n-1}), v \in Y(Q_{n-1})} f'_u([x_{0u},y_{1u}])
g'_v([x_{1v},y_{0v}]) \prod_{e=(s,t) \in E(Q_{n-1})}
w'_e([x_{0s},y_{1s}],[x_{1t},y_{0t}])\right|.
\end{eqnarray}
Then we apply the induction hypothesis to the right-hand side
of~(\ref{eq:Ind2}) and obtain

\begin{eqnarray*}
\mbox{R.H.S. of (\ref{eq:Ind2})} &\le& \prod_{e=(a,b) \in E(Q_{n-1})} \left(\sum \left|
\prod_{u \in X(Q_{n-1}), v \in Y(Q_{n-1})} f'_a([x_{0u},y_{1u}])
g'_b([x_{1v},y_{0v}]) \right.\right.\\ & &\left. \left.  \prod_{(s,t) \in E(Q_{n-1})}
w'_{(a,b)}([x_{0s},y_{1s}],[x_{1t},y_{0t}])\right| \right)^{1/|E(Q_{n-1})|} \\ &=&
\prod_{e=(a,b) \in E(Q_{n-1})} \left(\sum R_{\psi_e}
\right)^{1/|E(Q_{n-1})|},
\end{eqnarray*}
where for $e=(a,b)$
\begin{equation}
\begin{array}{lcr}
\psi_e(0u)=0a & \qquad & \forall u \in X(Q_{n-1}) \\
\psi_e(1u)=1a & \qquad & \forall u \in X(Q_{n-1}) \\
\psi_e(0v)=0b & \qquad & \forall v \in Y(Q_{n-1}) \\
\psi_e(1v)=1b & \qquad & \forall v \in Y(Q_{n-1})
\end{array}
\end{equation}
Combining this with (\ref{eq:Ind2}) we obtain
\begin{equation}
\label{eq:Qn-2} \mbox{L.H.S. of (\ref{eq:cube})} \le \prod_{e=(a,b) \in E(Q_{n-1})} \left(\sum R_{\psi_e}
\right)^{1/|E(Q_{n-1})|}.
\end{equation}

\noindent {\bf Step 2:} In this step we obtain a different bound for
the left-hand side of (\ref{eq:cube}). In Step 1, for every $v \in
Q_{n-1}$, we grouped the two vertices $0v$, $1v$ and the edge
between them as one vertex (see~(\ref{eq:Ind1})) and this reduced
$Q_n$ to $Q_{n-1}$. In this step we reduce $Q_n$ to $Q_2$. For every
vertex $s \in \{00,11\} =X(Q_2)$, define
$$f''_s= \left(\prod_{u \in X(Q_{n-2}), v \in Y(Q_{n-2})} f_{su} g_{sv}\right)\left( \prod_{(u,v) \in E(Q_{n-2})}
w_{(su,sv)}\right),$$
for every $t\in \{01,10\}=Y(Q_2)$, define
$$g''_t= \left(\prod_{u \in X(Q_{n-2}), v \in Y(Q_{n-2})} f_{tv} g_{tu}\right)\left( \prod_{(u,v) \in E(Q_{n-2})}
w_{(tv,tu)}\right),$$
and for every edge $e=(s,t) \in E(Q_2)$,
$$w''_e= \prod_{u \in X(Q_{n-2}), v \in Y(Q_{n-2})} w_{(su,tu)}w_{(tv,sv)}.$$
Note that the product in the left-hand side of~(\ref{eq:cube}) is
equal to
$$\left(\prod_{s \in X(Q_2), t\in Y(Q_2)} f''_s g''_t\right) \left( \prod_{(s,t) \in E(Q_2)} w''_{(s,t)}\right).$$
We can apply Lemma~\ref{lem:base} with proper index sets to these
functions. We get
\begin{equation}
\label{eq:Q2} \mbox{L.H.S of (\ref{eq:cube})} \le
\prod_{e=(s,t) \in E(Q_{2})} \left(\sum R_{\rho_e} \right)^{1/4},
\end{equation}
where for $e=(s,t)$
\begin{equation}
\begin{array}{lcr}
\rho_e(s'v)=sv & \qquad & \forall s' \in X(Q_2), v \in V(Q_{n-2}) \\
\rho_e(t'v)=tv & \qquad & \forall t' \in Y(Q_2), v \in V(Q_{n-2}) \\
\end{array}
\end{equation}

\noindent {\bf Step 3}: In this step we combine Steps 1 and 2. Note
that in~(\ref{eq:Qn-2}), the product $R_{\psi_e}$ has the same form
as the product in the left-hand side of (\ref{eq:cube}). Thus we can
apply Step 2 to $\sum R_{\psi_e}$. For $e \in E(Q_{n-1})$ we get

\begin{equation}
\sum R_{\psi_e} \le \prod_{e'=(s,t) \in E(Q_2)}\left(\sum
R_{\rho_{e'} \circ \psi_e}\right)^{1/4}.
\end{equation}
Combining this with (\ref{eq:Qn-2}) we obtain
\begin{equation}
\label{eq:step3} \mbox{L.H.S. of (\ref{eq:cube})} \le \prod_{e \in
E(Q_{n-1})}  \left( \prod_{e' \in E(Q_2)}\left(\sum R_{\rho_{e'}
\circ \psi_e}\right)^{1/4|E(Q_{n-1})|} \right)
\end{equation}

\noindent {\bf Step 4}: Now for some integer $k>0$ we repeatedly
apply Step 3, and by (\ref{eq:step3}) we get,
\begin{equation}
\label{eq:step4} \mbox{L.H.S. of (\ref{eq:cube})} \le
\prod_{e_1,\ldots,e_k \in E(Q_{n-1})} \left( \prod_{e'_1,\ldots,e'_k
\in E(Q_2)}\left(\sum R_{\rho_{e'_1} \circ \psi_{e_1} \circ \ldots
\circ \rho_{e'_k} \circ \psi_{e_k}
}\right)^{{4^{-k}|E(Q_{n-1})|}^{-k}}\right).
\end{equation}

Let us first assume that $\|w_e\|_\infty, \|f_u\|_\infty,
\|g_v\|_\infty <C$ for some constant $C>0$. We shall deal with the
general case later. Note first that for every arbitrary $\phi:V(Q_n)
\rightarrow V(Q_n)$, we have $\sum R_\phi < L$ for some large number
$L$ which depends on $C$, $f_u$'s, $g_v$'s, and $w_e$'s but does not
depend on $\phi$. Notice that for $e=(a,b) \in E(Q_{n-1})$
\begin{equation}
\label{eq:colapse1} \rho_{(00,01)} \circ \psi_e =\phi_{(0a,0b)},
\end{equation}
and
\begin{equation}
\label{eq:colapse2} \rho_{(11,10)} \circ \psi_e=\phi_{(1a,1b)},
\end{equation}
where $\phi_{(0a,0b)}$ and $\phi_{(1a,1b)}$ are defined as in
Remark~\ref{rem:homs}.

Next note that for every $\tilde{e} \in E(Q_n)$, $e \in E(Q_{n-1})$,
and $e' \in E(Q_2)$, we have $\rho_{e'} \circ \psi_e \circ
\phi_{\tilde{e}}=\phi_{\tilde{e}}$. Then from (\ref{eq:colapse1})
and (\ref{eq:colapse2}) we can conclude that whenever there exists
$1 \le i \le k$ such that $e'_i \in \{(00,01),(11,10)\}$, then
$\rho_{e'_1} \circ \psi_{e_1} \circ \ldots \circ \rho_{e'_k} \circ
\psi_{e_k}=\phi_e$ for some $e \in E(Q_n)$. Thus from
(\ref{eq:step4}), there exists numbers $p_e \ge 0$ such that
$$\sum_{e \in
E(Q_n)} p_e = 1-2^{-k},$$
and
\begin{equation}
\label{eq:step4-final} \mbox{L.H.S. of (\ref{eq:cube})} \le \prod_{e
\in E(Q_n)} \left(\sum R_{\phi_e}\right)^{p_e} L^{2^{-k}}.
\end{equation}
Since $Q_n$ is edge transitive, by applying the bound
(\ref{eq:step4-final}) to different permutations of the edges and
taking the geometric average, we finally conclude that
\begin{equation}
\label{eq:sym} \mbox{L.H.S. of (\ref{eq:cube})} \le L^{2^{-k}}
\prod_{e \in E(Q_n)} \left(\sum
R_{\phi_e}\right)^{(1-2^{-k})/|E(Q_n)|} .
\end{equation}
By tending $k$ to infinity, (\ref{eq:sym}) reduces to
(\ref{eq:cubeReform}).

\noindent {\bf Step 5}: Now consider the general case where
$\|f_u\|_\infty$, $\|g_v\|_\infty$, and $\|w_e\|_\infty$ need not be
bounded. Fix $C>0$ and let $f'_u := \max(f_u,C)$, $g'_v :=
\max(g_v,C)$ and $w'_e:=\max(w_e,C)$. We know that
Claim~\ref{claim:hypercubes} holds for these functions. By tending
$C$ to infinity the dominated convergence theorem implies the claim
for the general case as well.
\end{proof}

\subsection{Proof of Theorem~\ref{thm:criterion} \label{sec:criterion}}
Let $V(G)=X \cup Y$ be the bipartization of $G$, and denote
$m=|E(G)|$, and $n=|V(G)|$. Consider ${\cal I}=\{1,\ldots,k\}$ for
some $k > 1$.

\noindent {\bf (i):} Let $\lambda \in \mathbb{R}^k$ be such that
$\lambda_i \ge 0$ for $1 \le i \le k$. Define $w \in {\cal
W}^+({\cal I} \times {\cal I})$ as
$$w(x,y)=\left\{
\begin{array}{lcl}
\lambda_x& & \mbox{$x=y$}\\
0& & \mbox{otherwise}
\end{array}
\right.
$$
Note that $\|1\|_{r(G)}=k^{n/m}$, and $\|w\|_{r(G)}=\|\lambda\|_m$.
Now let $H$ be a subgraph of $G$ with $m'$ edges and $n'$ vertices,
and define $w_e=w$, if $e \in E(H)$, and $w_e=1$ otherwise. Then by
Theorem~\ref{thm:Holder},
$$k^{n-n'}\|\lambda\|_{m'}^{m'} = \sum \prod_{e \in E(G)} w_e  \le \|1\|_{r(G)}^{m-m'}\|\lambda\|_m^{m'}=
k^{n(m-m')/m}\|\lambda\|_m^{m'},$$
and so $$\|\lambda\|_{m'} k^{\frac{n}{m}-\frac{n'}{m'}}\le
\|\lambda\|_m.$$
 Since this holds for every $\lambda$, we have
$\frac{n}{m}-\frac{n'}{m'} \le \frac{1}{m}-\frac{1}{m'}$, or
$\frac{m'}{n'-1} \le \frac{m}{n-1}$.

\noindent {\bf (ii):} Let $w \in {\cal W}^+$ be defined as
$$w(x,y)=\left\{
\begin{array}{lcl}
1& & \mbox{$x=1$ or $y=1$}\\
0& & \mbox{otherwise}
\end{array}
\right.
$$
Notice that if $\prod_{(u,v) \in E(G)} w(x_u,y_v) \neq 0$ then $\{u:
x_u>1\} \cup \{v: y_v>1\}$ is an independent set. Thus denoting by
$I(G)$ the set of all independent sets of $G$, and by $\alpha(G)$
the size of its largest independent set, we have
$$h_G(w)=\sum_{S \in I(G)} (k-1)^{|S|}=C k^{\alpha(G)}+o(k^{\alpha(G)}),$$
where $C$ is the number of the independent sets of $G$ of size
$\alpha(G)$. Now let $S$ be a largest independent set of $G$, and
let $u \in V(G) \setminus S$. Without loss of generality assume that
$u \in X$. Let $H=(X,Y,E \setminus E(u,S))$, where $E(u,S)=\{(u,v)
\in E(G): v \in S\}$. Define $w_e=w$ if $e \in E(H)$, and $w_e=1$
otherwise. Note that
$$\sum \prod_{e=(u,v) \in E(G)} w_e \ge k (k-1)^{\alpha(G)},$$
and from Theorem~\ref{thm:Holder}, we get
$$k(k-1)^{\alpha(G)} \le k^{n\deg(u)/m} \left(C k^{\alpha(G)}+o(k^{\alpha(G)})\right)^{|E(H)|/m}.$$
Since this holds for every $k>1$, we get
$$k^{\alpha(G)+1} \le k^{n\deg(u)/m} k^{\alpha(G)|E(H)|/m},$$
and we get
\begin{equation}
\label{eq:crit:eq1} m \le (n - \alpha(G))\deg(u).
\end{equation}
 Since this holds for all $(n-\alpha(G))$ vertices that are in
$V(G) \setminus S$, we conclude that $V(G) \setminus S$ is an
independent set and so without loss of generality we may assume that
$S=Y$ and $|Y| \ge |X|$. Moreover (\ref{eq:crit:eq1}) implies that
all vertices in $X=V(G) \setminus S$ have the same degree
$m/(n-|Y|)$.

Next let $w_1 \in {\cal W}^+$ be defined as
$$w_1(x,y)=\left\{
\begin{array}{lcl}
1& & x,y=1 \\
0& & \mbox{otherwise}
\end{array}
\right.
$$
Then $\|w_1\|_G=1$. We showed above that $Y$ is the largest
independent set of $G$. Now consider $v \in Y$ of degree $d$. Let
$w_e=w_1$ for every edge $e$ incident to $v$, and let $w_e=w$ for
the rest of the $m-d$ edges. Then
$$\sum \prod_{e=(u,v) \in E(G)} w_e \ge (k-1)^{|Y|-1}.$$
Hence for every $k$,
$$(k-1)^{|Y|-1} \le
1^d\left(k^{|Y|}+o(k^{|Y|})\right)^{(m-d)/m},$$
which implies that $d|Y| \le m$. This shows that every vertex in $Y$
is of degree $d$.

\subsection{Proof of Theorem~\ref{thm:moduli}} Before determining the
moduli of smoothness and convexity of ${\cal W}_H$ we need the
following well-known technical lemma (see~\cite{Tomczak}).

\begin{lemma}
\label{lem:moduli:techn} Let $p \ge 2$, and $x,y  \in \mathbb{R}$.
Then
\begin{itemize}
\item[{\bf (i)}] We have $$(x+y)^p + (x-y)^p \le 2^{p-1}
(x^p+y^p).$$

\item[{\bf (ii)}] There exists a constant $K_p$ such that if $|y| \le 1$,
then

$$(1+y)^p + (1-y)^p -2 \le K_p |y|^2.$$
\end{itemize}
\end{lemma}
Now we can state the proof of Theorem~\ref{thm:moduli}.

\noindent
\begin{proof}[Theorem~\ref{thm:moduli}]
The inequalities $\delta_{{\cal W}_H} \le \delta_{\ell_m}$ and
$\rho_{\ell_m} \le \rho_{{\cal W}_H}$ follow from the fact that
$\ell_m$ is a subspace of ${\cal W}_H$. The key observation to prove
the theorem is that for $w_1,w_2 \in {\cal W}_H$, we have
\begin{equation}
\label{eq:moduli:key} \|w_1+w_2\|_H^m + \|w_1-w_2\|_H^m \le
(\|w_1\|_H+\|w_2\|_H)^m + (\|w_1\|_H-\|w_2\|_H)^m.
\end{equation}
To prove (\ref{eq:moduli:key}) expand the left-hand side. Some terms
will be canceled, and then use the fact that $H$ is H\"{o}lder to
bound each of the remaining terms. From (\ref{eq:moduli:key}) and
Lemma~\ref{lem:moduli:techn} {\bf (i)} we get
\begin{equation}
\label{eq:mud:conv1} \|w_1+w_2\|_H^m + \|w_1-w_2\|_H^m \le
2^{m-1}(\|w_1\|_H^m+\|w_2\|_H^m).
\end{equation}
Suppose that $\|w_1\|_H=\|w_2\|_H=1$ and $\|w_1-w_2\|_H \le
\epsilon$. Then by (\ref{eq:mud:conv1}) we have
$$\left\|w_1+w_2\right\|_H^m \le 2^m-\epsilon^m,$$
or
$$\left\|\frac{w_1+w_2}{2}\right\|_H \le \left(1-\left(\frac{\epsilon}{2}\right)^m\right)^{1/m}.$$
This shows that $\delta_{{\cal W}_H}(\epsilon) \ge 1 -
\left(1-\left(\frac{\epsilon}{2}\right)^m\right)^{1/m}$, and
finishes the proof of (\ref{eq:thm:convexity}) because it is known
(see \cite{MR540367}) that $\delta_{\ell_m}(\epsilon) \ge \epsilon^m
/m2^m+o(\epsilon^m)$.

Combining (\ref{eq:moduli:key}) and Lemma~\ref{lem:moduli:techn}
{\bf (ii)} we get that for $\|w_2\|_H \le \|w_1\|_H=1$,
\begin{equation}
\label{eq:moduli:smoothness1} \|w_1+w_2\|_H^m + \|w_1-w_2\|_H^m \le
K_m \|w_2\|_H^2 + 2.
\end{equation}
Next note that for $a\ge 0$ and $p \ge 1$ we have $a-1 \le
(a^p-1)/p$. From this we get
\begin{equation}
\label{eq:moduli:smoothness2} \|w_1+w_2\|_H + \|w_1-w_2\|_H -2 \le
m^{-1}(\|w_1+w_2\|_H^m + \|w_1-w_2\|_H^m - 2).
\end{equation}
Combining (\ref{eq:moduli:smoothness1}) and
(\ref{eq:moduli:smoothness2}) we have
\begin{equation}
\label{eq:moduli:smoothnessMain} \|w_1+w_2\|_H + \|w_1-w_2\|_H -2
\le m^{-1}K_m \|w_2\|_H^2.
\end{equation}
Thus for $0<\epsilon \le 1$, we have $\rho_{{\cal W}_H}(\epsilon)
\le m^{-1} K_m \epsilon^2$. This completes the proof of
(\ref{eq:thm:smoothness}) because it is known (see \cite{MR540367})
that $\rho_{\ell_m}(\epsilon) = (m-1) \epsilon^2/8 + o(\epsilon^2)$.
\end{proof}

\section{Concluding remarks and open questions}

\begin{itemize}
\item It is possible to generalize the framework of this article to hypergraphs
and define $\|\cdot\|_H$ norms when $H$ is a $k$-partite $k$-uniform
hypergraph. These norms will be defined on the space of the
functions $f:{\cal I}_1 \times \ldots \times {\cal I}_k \rightarrow
\mathbb{R}$. When $H$ is the complete $k$-partite $k$-uniform
hypergraph and every part is of size exactly $2$, we get the $k$-th
Gowers norm. Again one can ask that for which hypergraphs the
function is a norm.

\item Is there any edge transitive bipartite graph that is not
weakly H\"{o}lder?

\item In Theorem~\ref{thm:Hypercubes} {\bf (iii)} we showed that hypercubes are weakly H\"{o}lder.
We do not know the situation for any other graph that is of the form
of the Cartesian products of even cycles and single edges. As the
smallest case we suggest determining whether $K_2 \times C_6$ is
weakly H\"{o}lder or not.

\item Prove or disprove that  hypercubes are H\"{o}lder.

\item Consider the graph $H$ that is obtained by removing the edges of a Hamiltonian cycle from $K_{5,5}$. This is the
smallest graph for which Sidorenko's conjecture is
open~\cite{MR1225933}. Is this graph weakly H\"{o}lder? By
Theorem~\ref{thm:strongSid}, a positive answer verifies Sidorenko's
conjecture for this graph.

\item Determine the moduli of smoothness and convexity of $r(H)$
when $r(H)$ is a norm.
\end{itemize}

\section*{Acknowledgements}
The author wishes to thank Bal{\'a}zs Szegedy for introducing the
problem and pointing out the relation to Sidorenko's conjecture, and
also for many valuable discussions that made this work possible.
Also many thanks to L{\'a}szl{\'o} Lov{\'a}sz and Benny Sudakov for
short but valuable discussions.

\bibliographystyle{plain}
\bibliography{norm}
\end{document}